\documentclass[12pt]{article}

\usepackage{amssymb, amsmath, amsthm, graphicx}
\usepackage[left=1in,top=1in,right=1in]{geometry}

\usepackage{pgfplots}
\pgfplotsset{compat=1.15}
\usepackage{mathrsfs}
\usetikzlibrary{arrows}

\usepackage{algorithm}
\usepackage{algpseudocode}
\usepackage{url}

\theoremstyle{plain}
      \newtheorem{theorem}{Theorem}
      \newtheorem{lemma}[theorem]{Lemma}
      
      \newtheorem{problem}{Problem}
      
      \newtheorem{corollary}[theorem]{Corollary}
      \newtheorem{proposition}[theorem]{Proposition}
      
\theoremstyle{definition}
      \newtheorem{definition}{Definition}

\theoremstyle{remark}

\newcommand{\setbuilder}[2]{\left\{#1\colon#2\right\}}

\title{On the number of edges of restricted matchstick graphs}
\author{Panna Geh\'er\thanks{E\"otv\"os University, Budapest, Hungary and Alfr\'ed R\'enyi Institute of Mathematics, Budapest, Hungary. Email: \texttt{geher.panna@ttk.elte.hu}.}
\and J\'anos Pach\thanks{Alfr\'ed R\'enyi Institute of Mathematics, Budapest, Hungary. Email: \texttt{pach@renyi.hu}.}
\and Konrad Swanepoel \thanks{Department of Mathematics, London School of Economics and Political Science. \newline Email: \texttt{k.swanepoel@lse.ac.uk}.}
\and G\'eza T\'oth\thanks{Alfr\'ed R\'enyi Institute of Mathematics and
 Department of Computer Science and Information Theory,
Budapest University of Technology and Economics, Budapest, Hungary.
 Email: \texttt{geza@renyi.hu}.}}

\begin{document}

\maketitle

\begin{abstract}

A graph whose vertices are points in the plane and whose edges are noncrossing straight-line segments of unit length is called a \emph{matchstick graph}. We prove two somewhat counterintuitive results concerning the maximum number of edges of such graphs in two different scenarios.

First, we show that there are triangle-free matchstick graphs on $n$ vertices with more edges than the $\sqrt{n}\times\sqrt{n}$ grid.
On the other hand, every triangle-free matchstick graph on $n$ vertices has at most $2n-c\sqrt{n}$ edges for some $c > 0$.

Second, we prove that for every $r>0$, there is a constant $\varepsilon(r)>0$ with the property that, if $n$ is sufficiently large, every matchstick graph on $n$ vertices contained in a disk of radius $r$ has at most $(2-\varepsilon(r))n$ edges.

\end{abstract}



\section{Introduction}

Planar graphs have several interesting characterizations. One of them is given by the \mbox{so-called} circle packing theorem, also known as the Koebe--Andreev--Thurston theorem \cite{S05}. It states that any planar graph can be realized by a set of interior-disjoint circular disks in the following way: each disk represents a vertex, and two disks are tangent if and only if there is an edge connecting the two corresponding vertices of the graph. The maximum number of edges of a planar graph on $n \geq 3$ vertices is $3n- 6$, with the extremal graphs being triangulations (planar graphs in which every face is bounded by a cycle of length $3$). However, the set of disks realizing a triangulation (with more than three vertices) necessarily contains disks of different radii. This motivates the following definition. A graph is called a \emph{penny graph} if it can be realized by a set of unit disks. Equivalently, a penny graph is a graph that can be drawn in the plane such that all edges are drawn as unit segments and with no two vertices closer than unit distance.

How can one determine the maximum number of edges of a penny graph, in terms of its number of vertices? It is easy to see that every vertex has degree at most six, and also that not all vertices can be of degree six. If we call the vertices of degree smaller than six \emph{boundary points}, then it is not hard to argue, using the isoperimetric inequality, that there are at least $c\sqrt n$ boundary points, for some constant $c>0$ \cite{E46}.
Harborth \cite{Harborth74} proved that a penny graph on $n$ vertices can have at most $\lfloor 3n - \sqrt{12n - 3}\rfloor$ edges. His result is tight, as demonstrated by a hexagonal piece of the regular triangular lattice.

Harborth later conjectured \cite{Harborth2} that his result holds for the larger graph class of so-called matchstick graphs (introduced by him in 1981 \cite{Harborth1}). \emph{Matchstick graphs} are plane unit distance graphs; in other words, they can be drawn such that the vertices are represented by points in the plane, and the edges are non-crossing straight-line segments of unit length.
These graphs are much harder to handle than penny graphs.
For instance, unlike penny graphs, there is no upper bound for the maximum degree of the vertices of a matchstick graph, which is the main difficulty in proving Harborth's conjecture. After almost 40 years, the problem was recently settled by Lavoll\'ee and Swanepoel \cite{LS22}. The main idea of their proof is the following. If a matchstick graph has more than $3n-\sqrt{12n-3}$ edges, then almost all of its vertices must have degree $6$ and almost all faces must be triangular. Since all triangular faces are necessarily equilateral triangles, the area of the union of the bounded faces is bounded from below by $\Omega(n)$. By the isoperimetric inequality, the boundary of the union of the bounded faces has length at least $\Omega(\sqrt{n})$, which means that there must be $\Theta(\sqrt{n})$ boundary points. Then it is shown that in an optimal configuration, a large part of the vertex set must lie on a triangular lattice. Then a modified isoperimetric inequality (a form of L’Huilier’s inequality) is applied to show that the whole graph must lie on a triangular lattice.

In this paper we consider two further problems regarding matchstick graphs that turn out to be much harder than they appear, and the solutions are somewhat counterintuitive.

Penny graphs and matchstick graphs with a maximum number of edges contain many triangles, thus making it an intriguing problem to determine the maximum number of edges that \emph{triangle-free} penny graphs and matchstick graphs with $n$ vertices can have. For the case of penny graphs, this question was first considered by Swanepoel \cite{S09}. He conjectured that the maximal number of edges of a triangle-free penny graph on $n$ vertices is $\lfloor 2n - 2 \sqrt n \rfloor$. Apart from finding a non-trivial constant factor on the square-root term in \cite{E17}, giving an upper bound of $2n-1.65\sqrt{n}$, this conjecture is still open.
The analogous question for matchstick graphs is also open, but is of a different nature.

\begin{problem} \label{problem_trianglefree}
Determine or estimate the maximum number of edges $e(n)$ that a triangle-free matchstick graph on $n$ vertices can have.
\end{problem}

Again, using Euler's formula, we obtain that $e(n) \leq 2n-4$, and this bound is tight for triangle-free planar graphs. For matchstick graphs, we expect to obtain a better bound. In the absence of triangles, every face of a matchstick graph has at least four sides. Since there is no positive lower bound for the area of such faces, one cannot directly use  isoperimetric inequalities.
We might conjecture, as in the case of triangle-free penny graphs \cite{S09}, that, if~$n$ is a perfect square, then the maximum number of edges is attained  by a $\sqrt n \times \sqrt n$ piece of the integer lattice (see Figure \ref{fig_grid}), in which case the maximum would be $2n-2\sqrt{n}-O(1)$. Our first theorem shows that this is not the case. In particular, there is a better lower bound than that given by the square grid.

\begin{figure}[h!]
\centering

\begin{tikzpicture}[scale=0.7]
\foreach \x in {1,...,7} {
    \foreach \y in {1,...,7} {
        \fill (\x,\y) circle (4pt);
    }
}

\foreach \x in {1,...,7} {
\draw (\x, 1) --  (\x, 7); 
\draw (1, \x) --  (7, \x); 
}

\end{tikzpicture}

\caption{The $\sqrt n \times \sqrt n$ piece of the integer lattice defines a triangle-free matchstick graph with $n$ vertices and $2n-2\sqrt{n}-O(1)$ edges.}
\label{fig_grid}
\end{figure}

\begin{theorem}\label{thm1}
The maximum number of edges $e(n)$ in a triangle-free matchstick graph on $n$ vertices satisfies
\[ 2n-\sqrt{2}\sqrt{n}-O(1) \leq e(n)\leq 2n-\frac{\sqrt{2}}{5}\sqrt{n}.    \]
\end{theorem}

A different constraint that we can put on a matchstick graph is that it must be contained in some fixed bounded region, such as a disk of radius $r$.
The diameter of the best known constructions for Problem~\ref{problem_trianglefree} (see Figure~\ref{fig:tiling}) tends to infinity with the number of vertices.
However, for sufficiently large $n$, it is easy to construct triangle-free unit distance graphs with $n$ vertices and $(1+c)n$ edges in a disk of fixed radius $r$, where $c=c(r) > 0$ (see Proposition~\ref{prop:lowerbound}).
The following question arises.

\begin{problem} \label{problem_boundedmatchstick}
Let $e_{r}(n)$ be the maximum number of edges of a matchstick graph on $n$ vertices contained in a
disk of radius $r$. 
For a fixed $r>0$, determine $e_{r}(n)$ or at least its asymptotic behavior, as $n\rightarrow\infty.$ 
\end{problem}

Since a disk of radius $r$ has area $\pi r^2$ and a regular triangle of unit side has area $\sqrt{3}/4$,
a matchstick graph on $n$ vertices contained in a disk of radius $r$ has fewer than $8r^2$ triangular faces, all the remaining faces have at least four sides. Therefore, 
Euler's formula implies an upper bound of $2n+O_r(1)$ for the number of edges. 
Our second theorem shows that, somewhat surprisingly, the answer is substantially
smaller than $(2-o(1))n$.

\begin{theorem}\label{boundedmatchstick}
For any $r > 1$, there exist $c_1(r),c_2(r) > 0$ such that for any $n\geq 1$, 
we have
\[\left(2-c_1(r)\right)n \le e_r(n)\le \left(2-c_2(r)\right)n + O_r(1).\]
We can take 
$c_1(r)=1/(\lceil 2r\rceil-1)$ and 
$c_2(r)=1/(20\cdot3^{10r^2})$.
\end{theorem} 
The constant $c_2(r)$ is probably far from the truth, while $c_1(r)$ is probably closer to its optimal value.

\section{Triangle-free matchstick graphs}
In this section, we prove Theorem~\ref{thm1}.
We first consider the lower bound.

\begin{proposition}\label{prop:triangle-free-lower-bound}
For each $n\geq 1$ there exists a triangle-free matchstick graph on $n$ vertices with $\left\lfloor 2n-\sqrt{2n-\frac{7}{4}}-\frac{3}{2}\right\rfloor$ edges.
\end{proposition}

\begin{proof}
The statement is trivial for $n < 3$. For $n\geq 3$, choose $k\geq 1$ such that
\begin{equation}\label{ineq1}
1+\binom{k+1}{2} < n \leq 1+\binom{k+2}{2}.
\end{equation}
Choose unit vectors $u_1,u_2,\dots,u_k$ in clockwise order within some open semicircular arc of the unit circle in the plane. For any $i,j$ satisfying $1\leq i\leq j\leq k$, set $v_{i,j}=\sum_{t=i}^j u_t$
(see Figure~\ref{fig:tiling}).

Let \[V_k=\{o\}\cup\{v_{i,j}\mid 1\leq i\leq j\leq k\},\] where $o$ is the origin in the plane, and \[E_k=\setbuilder{ov_{i,i}}{i=1,\dots,k}\cup\setbuilder{v_{i,j}v_{i+1,j}}{1\leq i < j\leq k}\cup\setbuilder{v_{i,j}v_{i,j+1}}{1\leq i\leq j\leq k-1}.\]
Then it is easily checked that $G_k=(V_k,E_k)$ is a matchstick graph with $|V_k|=1+\binom{k+1}{2}$ and $|E_k|=k^2$.

Now let $G=(V,E)$ be a subgraph of $G_{k+1}$ induced by vertex set $V$ with 
\[V=V_k\cup\setbuilder{v_{i,k+1}}{1\leq i\leq n-1-\tbinom{k+1}{2}}\]
and then \[E=E_k\cup\setbuilder{v_{i,k+1}v_{i+1,k+1}}{1\leq i\leq n-2-\tbinom{k+1}{2}}\cup\setbuilder{v_{i,k}v_{i,k+1}}{1\leq i\leq n-1-\tbinom{k+1}{2}}.\]
We have $|V|=n$, $|E|=2n-k-3$.
It remains to show that $2n-k-3=\left\lfloor 2n-\sqrt{2n-\frac74}-\frac32\right\rfloor$.
For this, note that \eqref{ineq1} is equivalent to $\sqrt{2n-\frac74}-\frac32\leq k < \sqrt{2n-\frac74}-\frac12$, which is equivalent to $2n-\sqrt{2n-\frac74}-\frac52 < 2n-k-3\leq 2n-\sqrt{2n-\frac74}-\frac32$, that is, \[|E|=2n-k-3=\left\lfloor 2n-\sqrt{2n-\frac74}-\frac32\right\rfloor.\qedhere\]
\end{proof}

\begin{figure}[ht]
\begin{center}
\begin{tikzpicture}[line cap=round,line join=round,>=triangle 45,rotate=147,scale=0.75]
\draw [line width=1.5pt] (-6.645005805614367,5.3115128748036335)-- (-4.55,6.8);
\draw [line width=1.5pt] (-4.55,6.8)-- (-3.730015480361444,9.235624896199406);
\draw [line width=1.5pt] (-3.730015480361444,9.235624896199406)-- (-4.4982584629518705,11.688061636899114);
\draw [line width=1.5pt] (-4.4982584629518705,11.688061636899114)-- (-6.56128624004033,13.220562742410849);
\draw [line width=1.5pt] (-6.56128624004033,13.220562742410849)-- (-9.131092320514174,13.247764878225915);
\draw [line width=1.5pt] (-9.131092320514174,13.247764878225915)-- (-11.22609812612854,11.759277753029549);
\draw [line width=1.5pt] (-11.22609812612854,11.759277753029549)-- (-12.046082645767097,9.323652856830144);
\draw [line width=1.5pt] (-12.046082645767097,9.323652856830144)-- (-11.277839663176671,6.8712161161304355);
\draw [line width=1.5pt] (-11.277839663176671,6.8712161161304355)-- (-9.214811886088212,5.3387150106187);
\draw [line width=1.5pt] (-9.214811886088212,5.3387150106187)-- (-6.645005805614367,5.3115128748036335);
\draw [line width=1.5pt] (-12.046082645767097,9.323652856830144)-- (-9.95107684015273,10.81213998202651);
\draw [line width=1.5pt] (-9.95107684015273,10.81213998202651)-- (-9.131092320514174,13.247764878225915);
\draw [line width=1.5pt] (-9.131092320514174,13.247764878225915)-- (-8.3628493379,10.7953281375);
\draw [line width=1.5pt] (-6.299821560835288,9.262827032014473)-- (-8.3628493379,10.7953281375);
\draw [line width=1.5pt] (-9.182833857562304,8.359703241326802)-- (-8.3628493379,10.7953281375);
\draw [line width=1.5pt] (-7.068064543425715,11.71526377271418)-- (-9.131092320514174,13.247764878225915);
\draw [line width=1.5pt] (-7.068064543425715,11.71526377271418)-- (-4.4982584629518705,11.688061636899114);
\draw [line width=1.5pt] (-7.068064543425715,11.71526377271418)-- (-6.299821560835288,9.262827032014473);
\draw [line width=1.5pt] (-6.299821560835288,9.262827032014473)-- (-3.730015480361444,9.235624896199406);
\draw [line width=1.5pt] (-6.299821560835288,9.262827032014473)-- (-7.119806080473845,6.8272021358150665);
\draw [line width=1.5pt] (-9.95107684015273,10.81213998202651)-- (-9.182833857562304,8.359703241326802);
\draw [line width=1.5pt] (-9.182833857562304,8.359703241326802)-- (-7.119806080473845,6.8272021358150665);
\draw [line width=1.5pt] (-9.182833857562304,8.359703241326802)-- (-11.277839663176671,6.8712161161304355);
\draw [line width=1.5pt] (-9.214811886088212,5.3387150106187)-- (-7.119806080473845,6.8272021358150665);
\draw [line width=1.5pt] (-7.119806080473845,6.8272021358150665)-- (-4.55,6.8);
\draw [line width=1.5pt] (-6.645005805614367,5.3115128748036335)-- (-4.170932083018117,6.006928191425282);
\draw [line width=1.5pt] (-2.0759262774037506,7.495415316621648)-- (-1.2559417577651946,9.931040212821054);
\draw [line width=1.5pt] (-4.170932083018117,6.006928191425282)-- (-2.0759262774037506,7.495415316621648);
\draw [line width=1.5pt] (-1.2559417577651946,9.931040212821054)-- (-2.0241847403556212,12.383476953520763);
\draw [line width=1.5pt,dashed] (-2.0241847403556212,12.383476953520763)-- (-4.0872125174440805,13.915978059032497);
\draw [line width=1.5pt] (-4.55,6.8)-- (-2.0759262774037506,7.495415316621648);
\draw [line width=1.5pt] (-3.730015480361444,9.235624896199406)-- (-1.2559417577651946,9.931040212821054);
\draw [line width=1.5pt] (-4.4982584629518705,11.688061636899114)-- (-2.0241847403556212,12.383476953520763);
\draw [line width=1.5pt,dashed] (-6.56128624004033,13.220562742410849)-- (-4.0872125174440805,13.915978059032497);
\draw [line width=1.5pt,dashed] (-9.131092320514174,13.247764878225915)-- (-6.657018597917925,13.943180194847564);
\draw [line width=1.5pt,dashed] (-6.657018597917925,13.943180194847564)-- (-4.0872125174440805,13.915978059032497);

\draw [fill=black] (-6.645005805614367,5.3115128748036335) circle (2.5pt);
\draw[color=black] (-6.5,5) node {$v_{1,5}$};

\draw [fill=black] (-4.55,6.8) circle (2.5pt);
\draw[color=black] (-4.9,7.3) node {$v_{2,5}$};

\draw [fill=black] (-3.730015480361444,9.235624896199406) circle (2.5pt);
\draw[color=black] (-4.4,8.8) node {$v_{3,5}$};

\draw [fill=black] (-4.4982584629518705,11.688061636899114) circle (2.5pt);
\draw[color=black] (-5,11.2) node {$v_{4,5}$};

\draw [fill=black] (-6.56128624004033,13.220562742410849) circle (2.5pt);
\draw[color=black] (-7,12.85) node {$v_{5,5}$};

\draw [fill=black] (-9.131092320514174,13.247764878225915) circle (2.5pt);
\draw[color=black] (-9.4,13.5) node {$o$};

\draw [fill=black] (-11.22609812612854,11.759277753029549) circle (2.5pt);
\draw[color=black] (-12,11.5) node {$v_{1,1}$};

\draw [fill=black] (-12.046082645767097,9.323652856830144) circle (2.5pt);
\draw[color=black] (-12.8,9.3) node {$v_{1,2}$};

\draw [fill=black] (-11.277839663176671,6.8712161161304355) circle (2.5pt);
\draw[color=black] (-12,6.8) node {$v_{1,3}$};

\draw [fill=black] (-9.214811886088212,5.3387150106187) circle (2.5pt);
\draw[color=black] (-9.8,5.1) node {$v_{1,4}$};

\draw [fill=black] (-9.95107684015273,10.81213998202651) circle (2.5pt);
\draw[color=black] (-10.3,10.1) node {$v_{2,2}$};

\draw [fill=black] (-7.068064543425715,11.71526377271418) circle (2.5pt);
\draw[color=black] (-7.6,11.5) node {$v_{4,4}$};

\draw [fill=black] (-9.182833857562304,8.359703241326802) circle (2.5pt);
\draw[color=black] (-9.8,8.5) node {$v_{2,3}$};

\draw [fill=black] (-8.3628493379,10.7953281375) circle (2.5pt);
\draw[color=black] (-9,10.6) node {$v_{3,3}$};

\draw [fill=black] (-6.299821560835288,9.262827032014473) circle (2.5pt);
\draw[color=black] (-6.9,9.0) node {$v_{3,4}$};

\draw [fill=black] (-7.119806080473845,6.8272021358150665) circle (2.5pt);
\draw[color=black] (-7.9,6.7) node {$v_{2,4}$};

\draw [fill=black] (-4.170932083018117,6.006928191425282) circle (2.5pt);
\draw[color=black] (-3.7,5.9) node {$v_{1,6}$};

\draw [fill=black] (-2.0759262774037506,7.495415316621648) circle (2.5pt);
\draw[color=black] (-1.6,7.4) node {$v_{2,6}$};

\draw [fill=black] (-1.2559417577651946,9.931040212821054) circle (2.5pt);
\draw[color=black] (-0.7,10.2) node {$v_{3,6}$};

\draw [fill=black] (-2.0241847403556212,12.383476953520763) circle (2.5pt);
\draw[color=black] (-1.5,12.6) node {$v_{4,6}$};

\draw [fill=white] (-4.0872125174440805,13.915978059032497) circle (2.5pt);
\draw[color=black] (-3.9,14.4) node {$v_{5,6}$};

\draw [fill=white] (-6.657018597917925,13.943180194847564) circle (2.5pt);
\draw[color=black] (-6.399120961682941,14.4) node {$v_{6,6}$};

\begin{scope}[thick,xshift=9cm,yshift=5cm]
\draw[color=black,-latex] (-9.131092320514174,13.247764878225915) -- (-11.22609812612854,11.759277753029549) node[right] {$u_{1}$};
\draw[color=black,-latex] (-9.131092320514174,13.247764878225915) -- (-9.95107684015273,10.81213998202651) node[right] {$u_{2}$};
\draw[color=black,-latex] (-9.131092320514174,13.247764878225915) -- (-8.3628493379,10.7953281375) node[above] {$u_{3}$};
\draw[color=black,-latex] (-9.131092320514174,13.247764878225915) -- (-7.068064543425715,11.71526377271418) node[above] {$u_{4}$};
\draw[color=black,-latex] (-9.131092320514174,13.247764878225915) -- (-6.56128624004033,13.220562742410849) node[above left=-1mm] {$u_{5}$};
\draw[color=black,-latex] (-9.131092320514174,13.247764878225915) -- (-6.657018597917925,13.943180194847564) node[left] {$u_{6}$};
\draw [fill=black] (-9.131092320514174,13.247764878225915) circle (1.5pt) node[below] {$o$};
\end{scope}
\end{tikzpicture}
\end{center}
\caption{A triangle-free matchstick graph on $20$ vertices and $32$ edges ($k=5$).}\label{fig:tiling}
\end{figure}

We next prove the upper bound in Theorem~\ref{thm1}.
Consider a triangle-free matchstick graph~$G$ on $n$ vertices and $e$ edges, where $e$ is the maximum among all such graphs on $n$ vertices.
By maximality, $G$ is connected.
The upper bound is trivial if $n\leq 4$, so we assume that $n\geq 5$, and that each vertex has degree at least $2$, otherwise the bound follows by induction.
For a face of $G$, its length is the total length of its boundary walks, that is, the number of its boundary edges, where we count a boundary edge twice if the same face is on both of its sides. (Note that $G$ is not necessarily $2$-edge-connected.) Let $b$ denote the length of the unbounded face, and let $f_i$ denote the number of bounded faces of $G$ of length $i$, $i\ge 4$. Since~$G$ is connected, triangle-free and contains a cycle, $b \ge 4$.
Thus
\begin{equation}\label{eq:doublecounting}
2e= b+ \sum_{i\geq 4} if_i.
\end{equation}
If we use Euler's formula
\begin{equation}\label{eq:Euler}
n-e+\sum_{i\geq 4}f_i = 1
\end{equation}
to eliminate $f_4$ from \eqref{eq:doublecounting}, we obtain
\begin{equation}\label{eq:consequence}
e = 2n-4-\frac12F,
\end{equation}
where $F:=\sum_{i\geq 5}(i-4)f_i + (b-4)$.
We next find a lower bound for $F$ of the form $c\sqrt{n}$.
Note that each bounded face with $4$ boundary edges is a rhombus.
We next introduce rhombus chains.
A sequence $a_1b_1, \, a_2b_2, \, \dots, \, a_kb_k$ ($k\geq 2$) is a \emph{$v$-rhombus chain} if for each $i=1, \, \dots, \, k$, $a_ib_i$ is an edge of $G$, with $b_i-a_i$ all equal to the same unit vector $v$, and for each $i=1, \, \dots, \, k-1$, $a_ia_{i+1}$ and $b_ib_{i+1}$ are edges of $G$, such that the rhombi $a_ia_{i+1}b_{i+1}b_i$ are all positively oriented. See Figure \ref{rhombuschain}.

\begin{figure}[ht]
\begin{center}
\begin{tikzpicture}[line cap=round,line join=round,>=triangle 45,rotate=0,scale=0.75]

\draw [line width=1.5pt, -latex] (1, 0)-- (4, -1) node[midway, above] {$v$};
\draw [fill=black] (1, 0) circle (2.5pt); 
\draw [fill=black] (4, -1) circle (2.5pt);

\draw [line width=1.5pt] (1, 1)-- (4, 0);
\draw [fill=black] (1, 1) circle (2.5pt); \draw[color=black] (0.5,1) node {$a_{4}$};
\draw [fill=black] (4, 0) circle (2.5pt); \draw[color=black] (4.6,0) node {$b_{4}$};

\draw [line width=1.5pt] (3, 3)-- (6, 2);
\draw [fill=black] (3, 3) circle (2.5pt); \draw[color=black] (2.5,3) node {$a_{3}$};
\draw [fill=black] (6, 2) circle (2.5pt); \draw[color=black] (6.5,2) node {$b_{3}$};
\draw [line width=1.5pt] (4, 0)-- (6, 2);
\draw [line width=1.5pt] (1, 1)-- (3, 3);

\draw [line width=1.5pt] (4, 5)-- (7, 4);
\draw [fill=black] (4, 5) circle (2.5pt); \draw[color=black] (3.5,5) node {$a_{2}$};
\draw [fill=black] (7, 4) circle (2.5pt); \draw[color=black] (7.5,4) node {$b_{2}$};
\draw [line width=1.5pt] (3, 3)-- (4, 5);
\draw [line width=1.5pt] (6, 2)-- (7, 4);

\draw [line width=1.5pt] (3.5, 6)-- (6.5, 5);
\draw [fill=black] (3.5, 6) circle (2.5pt); \draw[color=black] (3,6) node {$a_{1}$};
\draw [fill=black] (6.5, 5) circle (2.5pt); \draw[color=black] (7,5) node {$b_{1}$};
\draw [line width=1.5pt] (4, 5)-- (3.5, 6);
\draw [line width=1.5pt] (7, 4)-- (6.5, 5);

\end{tikzpicture}
\end{center}
\caption{A $v$-rhombus chain $a_1b_1, \, a_2b_2, \, a_3b_3, \, a_4b_4$.}\label{rhombuschain}

\end{figure}

Note that if $a_1b_1, \, a_2b_2, \, \dots, \, a_kb_k$ is a $v$-rhombus chain, then $b_ka_k, \, b_{k-1}a_{k-1}, \, \dots, \, b_1a_1$ is a $(-v)$-rhombus chain, and we consider these two rhombus chains to be equivalent.
A $v$-rhombus chain $a_1b_1, \, a_2b_2, \, \dots, \, a_kb_k$ is \emph{maximal} if it is not contained in a $v$-rhombus chain $a_0b_0, \, a_1b_1, \, a_2b_2, \, \dots,a_kb_k$ or $a_1b_1, \, a_2b_2, \, \dots, \, a_kb_k, \, a_{k+1}b_{k+1}$.
Any rhombus face with sides parallel to unit vectors $u$ and $v$ is contained in exactly one maximal $u$-rhombus chain and exactly one maximal $v$-rhombus chain, up to equivalence.

Let $C$ be the number of all inequivalent maximal rhombus chains.
Each rhombus is contained in two maximal rhombus chains. Geh\'er and T\'oth \cite[Claim 5]{GT25} showed that a $u$-rhombus chain and a $v$-rhombus chain that are not contained in each other have at most one rhombus in common. (Similar ideas can be also found in \cite{KS92} and \cite{K93}.)
Since each bounded $4$-face is a rhombus, this implies the bound
\begin{equation}\label{eq:binom}
f_4\leq\binom{C}{2}.
\end{equation}
By maximality, we also have that for each maximal rhombus chain $a_1b_1, \, a_2b_2, \, \dots, \, a_kb_k$, the edges $a_1b_1$ and $a_kb_k$ are adjacent to bounded faces with at least
$5$ boundary edges or to the unbounded face. Hence, using $b \ge 4$,
\begin{equation}\label{eq:chain}
2C\leq\sum_{i\geq 5}if_i +b\leq 5\sum_{i\ge 5}(i-4) f_i + b= 5F + 20 - 4b \le 5F + 4.
\end{equation}
From \eqref{eq:binom} and \eqref{eq:chain}, we obtain 
\begin{equation}\label{eq:f_4upper}
f_4 \le \frac{(5F+4)(5F+2)}{8} = \frac{25}{8}F^2 + \frac{15}{4}F + 1.
\end{equation}
By \eqref{eq:Euler} and \eqref{eq:consequence}, and since
$\sum_{i\ge5} f_i \le \sum_{i\ge5}(i-4)f_i = F - b + 4$, we have
\begin{equation}\label{eq:f_4lower}
  f_4 = e - n + 1 - \sum_{i\ge 5} f_i
      \ge n - 7 - \frac32 F + b \ge n - 3 - \frac32 F .
\end{equation}
It follows from \eqref{eq:f_4upper} and  \eqref{eq:f_4lower} that
\[
  n \le \frac{25}{8}F^2 + \frac{21}{4}F + 4
    < \frac{25}{8}\Bigl(F + \frac{21}{25}\Bigr)^2 + 2,
\]
so
\[
  \frac{2\sqrt2}{5}\sqrt{n-2} - \frac{21}{25} < F,
\]
hence, by \eqref{eq:consequence} again and since
$\sqrt{n} - \sqrt{n-2} \le \sqrt{2}$,
\[
  e = 2n - 4 - \frac12 F
    < 2n - 4 + \frac{21}{50} - \frac{\sqrt2}{5}\sqrt{n-2}
    < 2n - \frac{\sqrt2}{5}\sqrt{n}.
\qed\]


\section{Bounded matchstick graphs}
The aim of this section is to prove Theorem~\ref{boundedmatchstick}.
We first establish the lower bound.

\begin{proposition}\label{prop:lowerbound}
Let $r > 1$ be given and let $D$ be a disk of radius $r$.
For each $n\geq 2r$ there exists a matchstick graph on $n$ vertices and more than $\left(2-\frac{1}{\lceil 2r\rceil-1}\right)n-\lceil 2r\rceil$ edges contained in $D$.
\end{proposition}

\begin{proof}
Let $p=\lceil 2r\rceil-1$.
Let $\varepsilon > 0$ and $k\geq 1$, both to be fixed later.
We next define a matchstick graph isomorphic to the $k\times p$ square grid with horizontal and vertical edges, but with the squares squashed such that the graph is close to a segment of length $p$.
For any $i,j$ satisfying $1\leq i\leq p$ and $1\leq j\leq k$, define the point
\[v_{i,j}=
\begin{cases}
    (i,j\varepsilon) & \text{if $j$ is odd,}\\
    (i+\sqrt{1-\varepsilon^2},j\varepsilon) & \text{if $j$ is even.}
\end{cases}\]
Let \[V_{p,k}=\setbuilder{v_{i,j}}{1\leq i\leq p, 1\leq j\leq k},\]
and \[E_{p,k}=\setbuilder{v_{i,j}v_{i+1,j}}{1\leq i < p, 1\leq j \leq k}\cup\setbuilder{v_{i,j}v_{i,j+1}}{1\leq i\leq p, 1\leq j < k}.\]
Then $G_{p,k}=(V_{p,k},E_{p,k})$ is a matchstick graph with $|V_{p,k}|=pk$ and $|E_{p,k}|=2pk-k-p$. See Figure \ref{konradconstruction}.
If we let $\varepsilon\to 0$, then $v_{i,j}\to(i,0)$ if $j$ is odd and $v_{i,j}\to(i+1,0)$ if $j$ is even.
Thus, in the limit as $\varepsilon\to 0$, the vertex set $V_{p,k}$ tends to the set $\{(1,0),(2,0),\dots,(p+1,0)\}$, which is contained in a disk of radius $p/2 < r$.
This implies that if we choose $\varepsilon$ sufficiently small depending on $p$ and $k$, the graph $G_{p,k}$ will be contained in a disk of radius $r$.

For general $n \geq 2r$, set $k=\lceil n/p\rceil-1\geq 1$.
Then $G=(V,E)$ is a subgraph of $G_{p, k+1}$ induced by 
\[ V=V_{p,k}\cup\setbuilder{v_{i,k+1}}{1\leq i\leq n-kp}\]
consequently  \[ E=E_{p,k}\cup\setbuilder{v_{i,k+1}v_{i+1,k+1}}{1\leq i < n-kp}\cup\setbuilder{v_{i,k}v_{i,k+1}}{1\leq i\leq n-kp}.\]
It is a matchstick graph in a disk of radius $r$ with $n$ vertices and \[2n-\lceil n/p\rceil-p > \left(2-\tfrac{1}{\lceil 2r\rceil-1}\right)n-\lceil 2r\rceil\] edges.
\end{proof}

\begin{figure}
\centering

\begin{tikzpicture}[thick,scale=0.7]
\foreach \x in {2,4,...,12} {
    \foreach \y in {1,...,7} {
        \draw[draw=black,fill=black] (\x,\y) circle (4pt);
    }
}

\foreach \x in {2,4,...,12} {
    \foreach \y in {1,...,6} {
        \draw [line width=1.2pt] (\x, \y)--(\x+1.7, \y+0.5);
        \draw [line width=1.2pt] (\x, \y+1)--(\x+1.7, \y+0.5);
    }
}
\foreach \x in {2,4,...,10} {
    \foreach \y in {1,...,7} {
        \draw [line width=1.2pt] (\x, \y)--(\x+2, \y);
    }
}
\foreach \x in {3,5,...,11} {
    \foreach \y in {1.5,...,6.5} {
        \draw [line width=1.2pt] (\x+0.7, \y)--(\x+2.7, \y);
    }
}

\foreach \x in {3,5,...,13} {
    \foreach \y in {1.5,...,6.5} {
        \draw[draw=black,fill=black!20!white] (\x+0.7,\y) circle (4pt);
    }
}

\end{tikzpicture}

\caption{The matchstick graph $G_{6, 13}$.}\label{konradconstruction}
\end{figure}

\begin{corollary}
For any $r > 1$ and any sufficiently large $n$, there exists a matchstick graph on~$n$ vertices and more than $\left(2-\frac{1}{\lceil 2r\rceil-1}\right)n + \frac{2\pi}{\sqrt{3}}r^2 - O(r)$ edges contained in $D$.
\end{corollary}
\begin{proof}
Note that the graph $G_{p,k}$ in Proposition~\ref{prop:lowerbound} is not only contained in $D$ for small enough~$\varepsilon$, but is in fact contained in a $p\times\varepsilon$ rectangle.
Also note that $D$ contains $\frac{2\pi}{\sqrt{3}}r^2-O(r)$ points of the unit triangular lattice forming the vertex set of a matchstick graph with $\frac{6\pi}{\sqrt{3}}r^2-O(r)$ edges.
We now combine these two constructions as follows.

Consider the construction of Proposition~\ref{prop:lowerbound} for $n-\frac{2\pi}{\sqrt{3}}r^2$ points in a $2r\times\varepsilon$ rectangle with centre coinciding with the centre of $D$.
Outside the rectangle, fill the circle up with points from the unit triangular lattice.
This gives us a matchstick graph inside $D$ with $n$ points and with
\[ \left(2-\frac{1}{\lceil 2r\rceil-1}\right)\left(n - \frac{2\pi}{\sqrt{3}}r^2\right) + \frac{6\pi}{\sqrt{3}}r^2 - O(r)\]
edges, which is what we had to prove.
\end{proof}

We continue with the upper bound in Theorem \ref{boundedmatchstick}.
Let $G$ be a matchstick graph on $n$ vertices in a disk $D$ of radius $r$.
Since a triangular face has area $\sqrt{3}/4$ and the area of $D$ is~${\pi r^2}$, 
there are at most $4\pi r^2/\sqrt{3}<8r^2$ triangular faces.
Remove an edge from the boundary of each triangle to obtain a triangle-free matchstick graph $G'$. 

Let $R$ be a rhombus face of $G'$.   
Suppose that its angles  are $\theta\le\pi/2$ and $\pi-\theta$. 
If $\theta\ge\pi/(12r)$, then it is called a \emph{fat rhombus}.

Each fat rhombus has area at least $\sin(\pi/(12r))\ge \pi/(20r)$. 
Therefore, there are at most ${20r^3}$ fat rhombi. 
Remove an edge from the boundary of each fat rhombus to obtain a graph $G''$ that does not contain a triangle or a fat rhombus, and 
we have ${e(G)\le e(G'')+20r^3+8r^2}$.
For the remainder of the proof, write $e=e(G'')$.

We define the \emph{neighborhood graph} $N$ of $G''$ by adding a vertex $v(\alpha)$ for each edge ${\alpha\in E(G'')}$, and connect two vertices $v(\alpha)$ and $v(\beta)$ in $N$ if $\alpha$ and~$\beta$ are on the boundary of 
the same face of $G''$.
Thus, we associate a face of $G''$ with some edges of $N$.
For any edges $\alpha$ and $\beta$ of $G''$, we say that they are \emph{neighbors} (resp.\ \emph{at distance $m$}) if $v(\alpha)$ and $v(\beta)$ are neighbors (resp.\ at distance $m$) in $N$.
(Note that $N$ is the line graph of the dual graph of $G''$.)

We say that an edge of $G''$ is \emph{irregular} if it is adjacent to a face which is not a rhombus, otherwise we call it \emph{regular}. 

\begin{lemma}\label{monopath}
Any regular edge has an irregular edge at distance at most $10r^2$ in $N$.
\end{lemma} 

\smallskip

Before proving Lemma \ref{monopath}, we finish the proof of Theorem \ref{boundedmatchstick}. 
Let $e^*$ denote the number of irregular edges of $G''$. As in the proof of Theorem~\ref{thm1}, let $b$ denote the length of the unbounded face of $G''$, and let $f_i$ denote the number of bounded faces of $G''$ of length $i$, $i\ge 4$.
Again, we have
\[2e=b+ \sum_{i\geq 4} if_i.\]
Since  $G''$ is not necessarily connected, here we have 
\[n-e+\sum_{i\geq 4}f_i \ge 1.\]
Consequently, we have
\begin{equation}\label{eq:consequence2}
e \le 2n-4-\frac12F,
\end{equation}
where $F=\sum_{i\geq 5}(i-4)f_i+(b-4)$, as before.
By the definition of irregular edges, \[e^*\leq b+ \sum_{i\geq 5}if_i.\]
From this and \eqref{eq:consequence2} it follows that
\begin{equation}\label{eq:e^*}
e\le 2n-4-\frac12F=2n-2-\frac{b}{2}-\frac{1}{2}\sum\limits_{i\geq 5}(i-4)f_i
< 2n-\frac{b}{10}-\frac{1}{2}\sum\limits_{i\geq 5}\frac{i}{5}f_i
\le 2n-\frac{e^*}{10}.
\end{equation}

For each regular edge $\alpha$, choose an irregular edge $\epsilon(\alpha)$ of minimum distance from $\alpha$ in~$N$.
By Lemma \ref{monopath}, we know that the distance between 
$\alpha$ and $\epsilon(\alpha)$ is at most $10r^2$, and in the shortest path between $v(\alpha)$ and $v(\epsilon(\alpha))$ in $N$ all edges correspond to rhombi.

For a fixed regular edge $\alpha$ we now estimate the number of edges $\beta$ for which 
${\epsilon(\beta)=\epsilon(\alpha)}$. 
For any such $\beta$ there is a path of length at most $10r^2$ from $\epsilon(\alpha)$ to $\beta$ 
all of whose edges correspond to rhombi. To build such a path, at each vertex we have three choices
to continue, corresponding to the other three sides of the rhombus. 
Thus, we have at most ${\sum_{i=1}^{10r^2}3^i<\frac32\cdot 3^{10r^2}}$ different paths ending at $\epsilon(\alpha)$. 
In other words, there are at most $\frac32\cdot 3^{10r^2}$ edges $\beta$ with $\epsilon(\beta)=\epsilon(\alpha)$.
It follows that
$\frac32\cdot 3^{10r^2}e^* \ge e-e^*$, so, since $3^{10r^2}\ge 2$, we can bound the number of irregular edges with
\[e^*\ge \frac{e}{2\cdot 3^{10r^2}}.\]
Consequently, by \eqref{eq:e^*}, with $K=20\cdot 3^{10r^2}$ we have $e<2n-e/K$, hence
\[e < \frac{2K}{K+1}\,n \le \left(2-\frac{1}{K}\right)n=\left(2-\frac{1}{20\cdot 3^{10r^2}} \right)n. \]
We conclude that
\[e(G)\le \left(2-\frac{1}{20\cdot 3^{10r^2}} \right)n+20r^3+8r^2.\]
\qed

\begin{proof}[Proof of Lemma \ref{monopath}]
Let $\alpha$ be a regular edge.
By rotating the drawing of $G''$, we can assume without loss of generality that $\alpha$ is horizontal.
We can also assume that there is no vertical edge of $G''$,
otherwise we rotate the picture by a very tiny angle.
Then each rhombus has a lower left, lower right, upper left and upper right edge. 
The \emph{angle} of an edge is its angle with the horizontal line.

A \emph{monotone path} is a path in $G''$ that intersects any vertical line at most once. 
Let $P$ be a monotone path in $G''$ and suppose that $R$ is a 
rhombus face whose lower left and lower right edges are in $P$. Then $R$ is called a \emph{hat} on $P$. 
If we replace the two lower edges of a hat by its two upper edges, then we obtain a new monotone path.
We refer to this procedure as \emph{flipping the hat}.

We say that an edge of a monotone path is \emph{right-sided} if the edge has a rhombus above it and it is the lower right side of the rhombus.
If the leftmost edge of a monotone path $P$ is right-sided, then we can extend $P$ to a new monotone path by adding the lower left side of the rhombus to the left end of $P$. See Figure \ref{hat+extend}.

\begin{figure}[ht]
\begin{center}
\begin{tikzpicture}[line cap=round,line join=round,>=triangle 45,rotate=0,scale=0.75]

\draw [fill=black] (0.6, 1.2) circle (2.5pt); 
\draw [fill=black] (2, 1) circle (2.5pt); 
\draw [fill=black] (2.6, 2.2) circle (2.5pt); 
\draw [fill=black] (4, 2) circle (2.5pt); 
\draw [line width=1pt, dashed] (0.6, 1.2)-- (2, 1) node[midway, below] {$f$};
\draw [line width=1.8pt] (2, 1)-- (4, 2) node[midway, below] {$e$};
\draw [line width=1pt, dashed] (0.6, 1.2)-- (2.6, 2.2);
\draw [line width=1pt, dashed] (2.6, 2.2)-- (4, 2);

\draw [fill=black] (5, 1.8) circle (2.5pt); \draw [line width=1.8pt] (4, 2)-- (5, 1.8);
\draw [fill=black] (6.4, 1.9) circle (2.5pt); \draw [line width=1.8pt] (5, 1.8)-- (6.4, 1.9);
\draw [fill=black] (8, 1) circle (2.5pt); \draw [line width=1.8pt] (6.4, 1.9)-- (8, 1);
\draw [fill=black] (10.2, 1.9) circle (2.5pt); \draw [line width=1.8pt] (8, 1)-- (10.2, 1.9);

\draw [fill=black] (8.6, 2.8) circle (2.5pt);
\draw [line width=1pt, dashed] (6.4, 1.9)-- (8.6, 2.8);
\draw [line width=1pt, dashed] (8.6, 2.8)-- (10.2, 1.9);
\draw[color=black] (8.3,1.8) node {\large $R$};

\draw [fill=black] (12, 2.4) circle (2.5pt); \draw [line width=1.8pt] (12, 2.4)-- (10.2, 1.9);

\draw[color=black] (6,0.7) node {\large $P$};

\end{tikzpicture}
\end{center}
\caption{A monotone path $P$ with hat $R$. Its leftmost edge $e$ is right-sided, thus $P$ can be extended by $f$.}\label{hat+extend}

\end{figure}

Similarly, we call an edge of a monotone path \emph{left-sided} if the edge has a rhombus above it and it is the lower left side of the rhombus.
If the rightmost edge of a monotone path $P$ is left-sided, then we can extend $P$ to a new monotone path by adding the lower right side of the rhombus to the right end of $P$.

We call the above two procedures \emph{extending $P$}.
Note that if a monotone path $P$ has neither an irregular edge nor a hat, then it can always be extended.
Indeed, in this case all the faces above the edges of $P$ are distinct rhombi.
If $P$ cannot be extended at its left end, then the leftmost edge of $P$ is not right-sided, so it must be left-sided. But then the second leftmost edge is also left-sided, since these two rhombi are disjoint. By the same argument we can conclude that all of the edges of $P$ must be left-sided, 
in particular the rightmost edge is also left-sided.
We can then extend $P$ at its right end. See Figure \ref{leftsided}.

\begin{figure}[ht]
\begin{center}
\begin{tikzpicture}[line cap=round,line join=round,>=triangle 45,rotate=0,scale=0.75]

\draw [fill=black] (2.2, 1.9) circle (2.5pt);

\draw [fill=black] (2, 1) circle (2.5pt); 
\draw [fill=black] (4.2, 2.9) circle (2.5pt); 

\draw [fill=black] (4, 2) circle (2.5pt); 
\draw [line width=1pt, dashed] (2.2, 1.9)-- (2, 1);

\draw [line width=1.8pt] (2, 1)-- (4, 2);

\draw [line width=1pt, dashed] (2.2, 1.9)-- (4.2, 2.9);
\draw [line width=1pt, dashed] (4.2, 2.9)-- (4, 2);

\draw [fill=black] (5, 1.8) circle (2.5pt); \draw [line width=1.8pt] (4, 2)-- (5, 1.8);

\draw [fill=black] (4.6, 3.4) circle (2.5pt);
\draw [fill=black] (5.6, 3.2) circle (2.5pt);
\draw [line width=1pt, dashed] (4, 2)-- (4.6, 3.4);
\draw [line width=1pt, dashed] (5, 1.8)-- (5.6, 3.2);
\draw [line width=1pt, dashed] (5.6, 3.2)-- (4.6, 3.4);

\draw [fill=black] (6.4, 1.9) circle (2.5pt); \draw [line width=1.8pt] (5, 1.8)-- (6.4, 1.9);
\draw [fill=black] (8, 1) circle (2.5pt); \draw [line width=1.8pt] (6.4, 1.9)-- (8, 1);
\draw [fill=black] (10.2, 1.9) circle (2.5pt); \draw [line width=1.8pt] (8, 1)-- (10.2, 1.9);
\draw [fill=black] (12, 2.4) circle (2.5pt); \draw [line width=1.8pt] (12, 2.4)-- (10.2, 1.9);

\draw [fill=black] (11.2, 2.9) circle (2.5pt);
\draw [fill=black] (13, 3.4) circle (2.5pt);
\draw [line width=1pt, dashed] (10.2, 1.9)-- (11.2, 2.9);
\draw [line width=1pt, dashed] (12, 2.4)-- (13, 3.4);
\draw [line width=1pt, dashed] (11.2, 2.9)-- (13, 3.4);

\draw[color=black] (6,0.7) node {\large $P$};

\end{tikzpicture}
\end{center}
\caption{If $P$ has no hat and the leftmost edge is left-sided, then all edges are left-sided and $P$ can be extended.}\label{leftsided}

\end{figure}

We consider the following procedure which builds a monotone path by starting with the horizontal edge $\alpha$ and repeatedly flipping hats and extending the current path.
The algorithm will stop when we find an irregular edge on the path.

\begin{algorithm}[ht]
\caption{Finding an irregular edge by extending a path of regular edges}
\begin{algorithmic}
\State{Let $P$ be the path of length $1$ consisting of the regular edge $\alpha$.}
\State $\ell \gets 1$
\State $d \gets 0$
\State $c \gets 0$
\While{$P$ does not have an irregular edge}
\If{$P$ has a hat}
\State{flip the hat}
\State $c\gets c-1$
\Else
\State {extend $P$}
\State $c\gets c+\ell$
\State $\ell\gets \ell+1$
\EndIf
\State $d\gets d+1$
\EndWhile
\end{algorithmic}
\end{algorithm}

Note that the algorithm only extends the path $P$ if $P$ does not have an irregular edge or a hat.
The counter $d$ counts the number of iterations of the while loop, and it is also an upper bound on the maximum distance in $N$ from $\alpha$ to an edge of $P$.
This is because both hat-flipping and path-extending procedure introduce new edges that are one rhombus further away from existing edges on the path.
Thus, we have to prove that the while loop has at most $10r^2$ iterations.

To do this, we first show that $\ell$, the length of $P$, never exceeds $4r$, and then that $c$ never becomes negative.

We first claim that $\ell$ is always bounded from above by $4r$.
When we flip a hat, we replace two edges of the path by two other, parallel  edges. 
If we extend a path, the new edge has an angle that differs by less than 
$\pi/(12r)$ from the angle of the edge of $P$ it is attached to.
Thus, if $P$ has length $\ell$, then each of its edges has an angle less than $(\ell-1)\pi/(12r)$.

Suppose that $\ell$ attains the value $\lfloor 4r\rfloor + 1$.
Then each edge of $P$ has an angle less than $4r\pi/(12r)=\pi/3$.
Since $\cos(\pi/3)=1/2$, the orthogonal projection of each edge of $P$ to the $x$-axis 
has length more than $1/2$. Therefore, the orthogonal projection of $P$ 
is longer than $\ell/2>4r/2=2r$, contradicting the fact that $P$ is in $D$. 
Thus, $\ell$ is bounded from above by~$\lfloor 4r\rfloor$.

Whenever we extend the path, we add $\ell$ to the counter $c$.
Thus, in total at most $\sum_{\ell=1}^{\lfloor 4r\rfloor-1} \ell = \binom{\lfloor 4r\rfloor}{2}$ is added to $c$ inside the \textbf{else} clause throughout the whole course of the algorithm.

Our second claim is that $c$ is always non-negative.

\begin{definition}
For any monotone path $P$, the \emph{convexity number} $c(P)$ of $P$ is defined to be the number of pairs of edges 
$(\epsilon_1, \epsilon_2)$ of $P$ with the property that 
$\epsilon_1$ is to the left of $\epsilon_2$, and~$\epsilon_1$ has smaller slope than $\epsilon_2$.
\end{definition} 

The counter $c$ in the algorithm is an upper bound to the convexity number of the current path.
Indeed, when we flip a hat, $c(P)$ decreases by $1$, and when we extend $P$, $c(P)$ increases by at most the length of the path before it is extended, which equals to the counter $\ell$ keeping track of the length of the current path.
Since $c(P)\geq 0$, it follows that the counter $c$ never becomes negative.

Since at most $\binom{\lfloor 4r\rfloor}{2}$ is added to $c$ in the \textbf{else} clause, it follows that $c$ is decreased by~$1$ at most the same number in the \textbf{then} clause.
Therefore, the \textbf{then} clause is executed at most~$\binom{\lfloor 4r\rfloor}{2}$ times.
Since the \textbf{else} clause is executed at most $4r$ times, the while loop is executed at most \[\binom{\lfloor 4r\rfloor}{2} + \lfloor 4r\rfloor \leq \frac{4r(4r-1)}{2}+4r < 10r^2\] times.
It follows that an irregular edge will be found at distance at most $10r^2$ from $\alpha$, which finishes the proof of Lemma~\ref{monopath}.
\end{proof}

\smallskip

\section{Concluding remarks}
\begin{enumerate}

\item Recently, Winkler~\cite{W26b} improved our upper bound on the maximum number $e(n)$ of edges of a triangle-free matchstick graph on $n$ vertices to $e(n)\le 2n-\frac{4\sqrt{5}}{25}\sqrt{n}$. We conjecture that $e(n)\le 2n-\sqrt{2n}+O(1)$, maybe even $e(n)=\left\lfloor 2n-\sqrt{2n-\frac{7}{4}}-\frac{3}{2}\right\rfloor$, the lower bound in Proposition~\ref{prop:triangle-free-lower-bound}.

\item
Euler's formula implies that in a planar graph of 
$n$ vertices and 
girth $g$, the number of edges $e$ is bounded from above by $e\le \frac{g}{g-2}(n-2)$.
It is natural to ask whether we have the stronger bound
$e\le \frac{g}{g-2}(n-2)-c(g)\sqrt{n}$ for matchstick graphs of girth $g\geq 5$, where $c(g) > 0$.
This is in fact easy to prove for {\em odd} 
$g$ using the isoperimetric inequality, since it is known that a simple polygon with an odd number of sides, all of unit length, has an area not smaller than that of a unit-length equilateral triangle~\cite{BKM99}.
On the other hand, it was observed by N\'ora Frankl \cite{F25}, that for any {\em even} $g\ge 6$, there exists a matchstick graph of girth $g$ with exactly $e=\frac{g}{g-2}(n-2)$ edges, if $n-2$ is divisible by $\frac{g}{2}-1$. Take two points, and sufficiently many noncrossing paths between them, each of length $g/2$. Recently, Winkler~\cite{W26a} gave other constructions for $g=6$ and $g=8$.

\item 
The upper bound in Theorem~\ref{boundedmatchstick} can surely be improved.
Is it true that $e_r(n) \le (2-c(r))n$ for some $c(r)=\Theta(1/r)$?
\end{enumerate}

\section*{Acknowledgments}
Panna Geh\'er, J\'anos Pach and G\'eza T\'oth were supported by the ERC Advanced Grant `GeoScape' No.\ 882971 and by the National Research, Development and Innovation Office, NKFIH, K-131529. Panna Geh\'er was also supported by the Thematic Excellence Program TKP2021-NKTA-62 of the National Research, Development and Innovation Office. J\'anos Pach was also supported by NSF Grant DMS-1928930, while he was in residence at SLMath Berkeley during the Spring semester of 2025 and by a Clay Senior Scholar fellowship at the IAS/Park City Mathematics Institute in July 2025.
Konrad Swanepoel was partially supported by ERC Advanced Grant `GeoScape' No.\ 882971 and by the Erd\H os Center.
Panna Geh\'er and G\'eza T\'oth were also supported by the National Research, Development and Innovation Office, 
ADVANCED 152590.

\end{document}